\documentclass[12pt]{amsart}

\usepackage{enumitem}
\usepackage{amssymb}
\usepackage{amsmath}
\usepackage{amsthm}
\usepackage{hyperref}
\usepackage{cleveref}

\usepackage{soul}

\usepackage{cases}

\usepackage{tikz}
\usetikzlibrary{decorations.markings, decorations.pathreplacing, positioning,arrows, matrix, decorations.pathmorphing, shapes.geometric, calc}
\usetikzlibrary{backgrounds, decorations}
\usepackage{tikz-cd}

\newtheorem*{notation}{Notation}

\newcounter{alancomments}

\newcounter{olicomments}
\newcommand{\oli}[1]{\textbf{\color{red}(O\arabic{olicomments})} \marginpar{\scriptsize\raggedright\textbf{\color{red}(O\arabic{olicomments})Oli:} #1}
\addtocounter{olicomments}{1}}

\newtheorem{theorem}{Theorem}[section]
\newtheorem{lemma}[theorem]{Lemma}
\newtheorem{corollary}[theorem]{Corollary}
\newtheorem{question}[theorem]{Question}

\newtheorem{proposition}[theorem]{Proposition}
\newtheorem{claim}[theorem]{Claim}
\newtheorem{thmletter}{Theorem}

\newtheorem{quesletter}[thmletter]{Question}
\theoremstyle{definition}

\newtheorem{definition}[theorem]{Definition}

\newcommand{\rk}{\operatorname{rank}}
\newcommand{\sgn}{\operatorname{sgn}}
\newcommand{\im}{\operatorname{im}}
\newcommand{\aut}{\operatorname{Aut}}
\newcommand{\inn}{\operatorname{Inn}}
\newcommand{\fix}{\operatorname{Fix}}
\newcommand{\stab}{\operatorname{Stab}}
\newcommand{\ab}{\operatorname{Ab}}

\newcommand{\BS}{\operatorname{BS}}

\newcommand{\bs}{\operatorname{BS}}

\newcommand{\p}[1]{\noindent {\newline\bf #1.}}

\title{Fixed subgroups of generalised Baumslag-Solitar groups}
\author{Oli Jones and Alan Logan}

\begin{document}

\maketitle

\begin{abstract}
    We investigate fixed subgroups of automorphisms of generalised Baumslag-Solitar (GBS) groups. 
    Our main results are for automorphisms leaving a Bass-Serre tree invariant, under the assumption that all edge stabilisers are strictly contained in the corresponding vertex stabilisers. We completely characterise which GBS groups admit such an automorphism with a fixed subgroup which is not finitely-generated. In doing so, we provide an infinite family of examples of non-finitely generated fixed subgroups in GBS groups. 
    
    Dropping the above assumptions, we show that all finite order automorphisms of GBS groups have finitely generated fixed subgroups. Furthermore, we show that when the GBS graph is a tree, all automorphisms have finitely generated fixed subgroups.
\end{abstract}

\section{Introduction}

Let $\phi$ be an automorphism of a group $G$.
Then the set of elements of $G$ fixed by $\phi$ form a subgroup, $\fix(\phi)$, called the \emph{fixed subgroup of $\phi$}, that is,
\[
\fix(\phi)=\{g\in G\mid\phi(g)=g\}
\]
Such subgroups have been studied for many groups, with the main topic of study being their rank, $\rk(\fix(\phi))$, which is the minimal cardinality of a generating set of the subgroup.
The main questions are, when are fixed subgroups of \emph{finite} rank? When are they of \emph{bounded} rank?

The Scott Conjecture, first studied in the 1970s, dealt with these questions for finitely generated free groups.
Gersten showed that fixed subgroups of free group automorphisms are always of finite rank \cite{Gersten1987ScottConj}.
Bestvina and Handel further proved that the rank is bounded, specifically $\rk(\fix(\phi))\leq \rk(F)$ for all $\phi\in\aut(F)$ \cite{TrainTracks}.
See \cite{Survey} for a survey on fixed subgroups of free groups.
Similar results have been obtained for, among other classes,
surface groups \cite{Surface},
hyperbolic \cite{Hyperbolic} and 
relatively hyperbolic groups \cite{RelHyperbolic},
$3$-manifold groups \cite{ThreeManifold}, and certain Artin groups \cite{jones2024fixedsubgroupsartingroups}. 
There are also results for some free products \cite{FreeProducts}, finite order automorphisms of certain direct products \cite{FreeAbelianTimesFree}, and some graphs of groups \cite{sykiotis2002GraphOfGroups}.
On the other hand, there are examples of biautomatic and $\operatorname{CAT}(0)$ which have infinitely many isomorphism classes of fixed subgroups \cite{BiautomaticCATZero}. There are automorphisms of RAAGs with non-finitely generated fixed subgroups, but this does not occur for ``untwisted'' automorphisms \cite{fioravanti2024coarse}.

The study of fixed subgroups has close links to Reidemeister fixed point theory \cite{NielsenFPT}, and has been applied to resolve the conjugacy problem for certain classes of extensions of groups \cite{ConjProb1, ConjProb2, ConjProb3}.

This paper studies the ranks of fixed subgroups of automorphisms of generalised Baumslag-Solitar groups. In contrast to settings with negative curvature, we find many examples of automorphisms of GBS groups with infinite rank fixed subgroups (in particular, in Lemma \ref{lem:Betti2InfiniteGen} we give an infinite family of examples). As such, our focus is on characterising which GBS groups have the good behaviour of finitely or boundedly generated fixed subgroups.

To state our results in full generality we will need more terminology. We begin by highlighting a corollary for the non-solvable classical Baumslag-Solitar groups.

\begin{thmletter}[Corollary \ref{cor:baumslagsolitar}]\label{thmintro:baumslagsolitar}
    For fixed $p, q \in \mathbb{Z}$ with $|q| \geq |p|$ and $|p| \neq 1$, consider the group $\bs(p,q) = \langle x, t \mid x^p = tx^qt^{-1} \rangle$.

    \begin{enumerate}
        \item If $p = -q$ then, for all $\phi \in \aut(\bs(p,q))$, $\rk(\fix(\phi)) \leq 3$.
        \item If $p \nmid q$ then, for all $\phi \in \aut(\bs(p,q))$, $\rk(\fix(\phi))$ is finite, but there is no bound on the rank.
        \item Otherwise, there exists $\phi \in \aut(\bs(p,q))$ such that $\rk(\fix(\phi))$ is infinite.
    \end{enumerate}
\end{thmletter}

We remark that it immediately follows, for example, that $\bs(2,3)$ has infinitely many isomorphism types of fixed subgroup.

\p{GBS systems}
A \emph{GBS system} $(G, T)$ is a finitely generated group $G$ and a tree $T$, along with an action $G\curvearrowright T$ where all edge and vertex stabiliers are infinite cyclic.
A \emph{Generalised Baumslag--Solitar (GBS) group} is a group $G$ which is part of a GBS system.
A GBS group is \emph{elementary} if it is isomorphic to $\mathbb{Z}$, $\mathbb{Z}^2$ or the Klein bottle group, and otherwise is \emph{non-elementary}.

A GBS system $(G,T)$ can be conveniently written as a graphs of groups $T//G$. Since all of the edge and vertex groups are $\mathbb{Z}$, the only data to record is the inclusion maps. Any homomorphism from $\mathbb{Z}$ to itself is multiplication by an integer, so we regard $T//G$ as a graph with the ends of edges labelled by non-zero integers.

%GBS systems are graphs of groups, and so may be defined by the graph $T/G$ and stating inclusion maps of edge groups into vertex groups in this graph.
%As all stablisers are isomorphic to $\mathbb{Z}$,
%we may replace the inclusion maps with a number from $\mathbb{Z}\setminus\{0\}$. We use $T//G$ to denote the corresponding $\mathbb{Z}$-labelled graph.

For a GBS group $G$, an automorphism $\phi\in\aut(G)$ is an automorphism of the GBS system $(G, T)$, written $\phi \in \aut^T(G)$, if $\phi$ leaves the tree $T$ invariant.
We give a precise definition in \Cref{sec:CompatibleAutomorphismsPreliminaries}.

\p{Finite generation}
A GBS system is \emph{$1$-free} if every edge stabiliser is a proper subgroup of both adjacent vertex stabiliers or, equivalently, the corresponding $\mathbb{Z}$-labelled graph has no label $\pm 1$.
Our main result classifies finite generation of fixed subgroups in this setting.

Its statement uses two group invariants of non-elementary GBS groups.
Firstly, $\beta(G)$ is the first Betti number of the quotient graph $T/G$, which is defined as $1-|V(T/G))|+|E(T/G))|$.
This 
%is equal to the rank of the fundamental group of the (topological) graph $T/G$, and 
turns out to be a property of the GBS group, independent of the choice of system $(G, T)$ (see Lemma \ref{lem:bettiIndependence}).
Secondly, $\Delta(G)$ is the modulus of $G$, as defined in \Cref{sec:GBSDfns}, and is a subgroup of $\mathbb{Q}^*$ which can be easily computed by looking at the loops in $T//G$.

\begin{thmletter}[\Cref{thm:characterisingFg}]\label{thmintro:characterisingFg}
    Suppose $(G,T)$ is a $1$-free, non-elementary GBS system.
    Then $\fix(\phi)$ is finitely generated for all $\phi \in \aut^T(G)$ if and only if one of the following occurs:

    \begin{enumerate}
        %\item $T/G$ has Betti number 2,
        \item $\beta(G)=0$, or
        %\item $T/G$ has Betti number 1 and $\Delta(G)$ is generated by an integer not equal to $-1$.
        \item $\beta(G)=1$ and either $\Delta(G) = \{1, -1\}$ or $\Delta(G)$ is not generated by an integer.
    \end{enumerate}
\end{thmletter}

The proof of \Cref{thmintro:characterisingFg} has two stages.
Firstly, we give sufficient conditions on $\phi \in \aut^T(G)$ which imply that $\fix(\phi)$ is finitely generated. These are summarised in \Cref{thm:sufficientFg}, although the most complex case is treated separately in Proposition \ref{prop:BettiNumber1fgFix}. 
Notably, for this direction we do not need to assume that the underlying GBS system $(G, T)$ is $1$-free.
Secondly, in \Cref{sec:nonFG_examples} we work under the $1$-free assumption and give explicit examples of automorphisms with non-finitely generated fixed subgroups in the relevant cases.

One notable case of Theorem \ref{thm:sufficientFg} is that $\phi$ is of finite order.
Every finite order automorphism preserves some GBS tree $T$ \cite{guirardel2007deformation}, so in this case we can drop the restriction to $\aut^T(G)$.

\begin{thmletter}[\Cref{cor:allFiniteOrder}]\label{thmintro:allFiniteOrder}
    Let $G$ be a GBS group, and $\phi \in \aut(G)$ be of finite order. Then $\fix(\phi)$ is finitely generated.  
\end{thmletter}

Note that \Cref{thmintro:allFiniteOrder} requires the automorphism $\phi$ to be of finite order, not just the outer automorphism class $[\phi]$.
To see this is necessary, consider $\BS(1, n)=\langle a, t\mid t^{-1}at=a^n\rangle$ for $|n|>1$, which splits as $\mathbb{Z}[1/n]\rtimes\mathbb{Z}$. Take $\phi$ to be conjugation by $a$, so $[\phi]$ has finite order in $\operatorname{Out}(\BS(1, n))$ since $\phi$ is inner.
However, $\fix(\phi)$ is the centraliser of $a$, which is the factor $\mathbb{Z}[1/n]$ in the semidirect product decomposition, and is non-finitely generated.

\p{Bounded Generation}
Theorem \ref{thmintro:characterisingFg} classifies finite generation under certain assumptions.
Working under the same assumptions, we now classify when there is a bound on the rank of the fixed subgroups.

\begin{thmletter}[\Cref{thm:bounded}]\label{thmintro:bounded}
    Suppose $(G,T)$ is a 1-free, non-elementary GBS system. Then, for $\phi \in \aut^T(G)$:

    \begin{enumerate}
        \item If $\beta(G) = 0$, then $\rk(\fix(\phi)) \leq \max(1, 2|E(T/G)|)$.
        \item If $\beta(G) = 1$ and $\Delta(G) = \{1, -1\}$, then $\rk(\fix(\phi)) \leq 2|V(T/G)| + 1$. 
    \end{enumerate}
Otherwise, there is no bound on $\rk(\fix(\phi))$.
    
\end{thmletter}

We remark that for certain GBS groups, those which are \emph{algebraically rigid}, there is essentially only 1 GBS system, and $\aut^T(G) = \aut(G)$ (see \cite{levitt2007gbs} for the graphical characterisation of algebraic rigidity). Thus Theorem \ref{thmintro:characterisingFg} and Theorem \ref{thmintro:bounded} are a complete classification for algebraically rigid GBS groups where the sole reduced system is 1-free. It is essentially this observation that yields Theorem \ref{thmintro:baumslagsolitar}.

\p{Arbitrary group automorphisms}
The conditions in \Cref{thmintro:characterisingFg} are group invariants, meaning that \Cref{thmintro:characterisingFg} applies to all automorphisms of any 1-free GBS system of a GBS group $G$.
However, one may wonder what happens when the restriction to $\aut^T(G)$ is lifted.
It would be interesting to know if \Cref{thmintro:characterisingFg} gives a complete classification of when GBS groups have only finitely generated fixed subgroups. 

In the setting where $\beta(G) = 0$, we prove that all fixed subgroups are finitely generated.
There is no restriction to only automorphisms of the GBS system, or regarding $1$-freeness.

\begin{thmletter}[\Cref{thm:treeFg}]\label{thmintro:treeFg}
    Suppose $G$ is a non-elementary GBS group with $\beta(G) = 0$.
    Then for all $\phi \in \aut(G)$, $\fix(\phi)$ is finitely generated.
\end{thmletter}

The proof of \Cref{thmintro:treeFg} is quite different to that of \Cref{thmintro:characterisingFg}, using BNS-invariants. However, these tools do not extend to the general case, leaving the following question to answer.

\begin{quesletter}
    \label{Question:MainQuestion}
    Let $G$ be a GBS group such that $\beta(G) = 1$ and $\Delta(G)$ is not generated by an integer other than $-1$. Does there exist $\phi \in \aut(G)$ such that $\fix(\phi)$ is not finitely generated?
\end{quesletter}

For now, this question seems hard to approach since we do not have a full picture of what $\aut(G)$ looks like for GBS groups; in comparison, the subgroups $\aut^T(G)$ are well understood \cite{levitt2007gbs}. We remark that automorphisms not in some $\aut^T(G)$ are precisely those acting without a fixed point on the associated \emph{deformation space}, which is analogous to outer space for $\mathrm{Out}(F_n)$. It is plausible that the right tool for analysing fixed subgroups of such automorphisms is a suitable analogue of train tracks for graphs of groups; train tracks for graphs were the key technical tool in \cite{TrainTracks}. 
See, for example, Lyman's work for results in this direction \cite{lyman2022train}.

\section*{Acknowledgements}

We are grateful to Gilbert Levitt for suggesting the proof of Theorem \ref{thm:treeFg}. The first named author would like to thank Yassine Guerch and Stefanie Zbinden for interesting discussions, and their supervisor Laura Ciobanu for her support.

%\section{Preliminaries and Compatibility}
\section{Preliminaries: GBS groups and Compatibility}

This paper revolves around Bass-Serre Theory, which is the theory of groups acting on trees.
We therefore now give preliminaries on this theory, we define GBS groups using actions on trees, as well as giving certain properties applied later, and we define and briefly study compatible automorphisms.

\subsection{Bass-Serre Theory}

A tree is a simply connected graph, and a tree with an action of a group $G$ is a $G$-tree. We assume the reader is familiar with the basic theory of group actions on trees, see for instance Serre's book \cite{trees}. For $s$ an edge or vertex of a $G$-tree, we write $G_s$ for the $G$-stabiliser. For $X \subseteq T$ (which will always be a subtree in practice), we will write $\stab(X)$ for the setwise stabiliser.
For a $G$-tree $T$ and en element $g \in G$, we write $\ell(g) := \inf \{d(x,gx) \mid x \in T\}$.
%Given $g \in G$ a group acting on a tree $T$, we will write $\ell(g) = \inf \{d(x,gx) \mid x \in T\}$.

%The following map from a fundamental group of a graph of groups to a free group will be used several times throughout the paper.

\p{Forgetful map}
The forgetful map goes from a fundamental group of a graph of groups to a free group, and is used several times throughout the paper:
%\begin{definition}\label{def:projectinMap}
    Given a group $G$ acting on a tree $T$, and $v \in T$, the \emph{forgetful map} is $p_*: G \rightarrow \pi_1(T/G, p(v))$ given by $g \mapsto \overline{p([v, gv])}$, where $p: T \rightarrow T/G$ is the obvious projection.
%\end{definition}
It is not hard to check that $\ker(p_*)$ does not depend on the choice of $v$.

\p{Minimal actions}
The action of a group $G$ on a tree $T$ is \emph{minimal} if there is no proper, non-empty, $G$-invariant subtree.
We require a result of Bass, which uses minimal actions to determine whether a group acting on a tree is finitely generated or not.

\begin{proposition}\cite[Proposition 7.9]{bass1993covering}\label{prop:bassFg}
    Let $G$ be a group acting on a tree $T$.

    \begin{enumerate}
        \item If the vertex stabilisers are finitely-generated and $T/G$ is finite, then $G$ is finitely generated.
        \item If the action is minimal and $G$ is finitely generated, then $T/G$ is finite.
    \end{enumerate}
\end{proposition}

\subsection{Generalised Baumslag--Solitar (GBS) groups}
\label{sec:GBSDfns}

Recall from the introduction that a \emph{GBS system} $(G, T)$ encodes the action $G\curvearrowright T$ of a finitely generated group $G$ on a tree $T$, where all edge and vertex stabiliers are infinite cyclic, and that a \emph{GBS group} is a group $G$ which is part of a GBS system.
%A GBS group is \emph{elementary} if it is isomorphic to $\mathbb{Z}$, $\mathbb{Z}^2$ or the Klein bottle group, and otherwise is \emph{non-elementary}.

\p{The Betti number of a GBS group}
Given a $G$-tree $T$, we will write $T/G$ for the quotient graph, which is not yet a graph of groups and depends only on $(G,T)$.
The \emph{Betti number} of a GBS system is the Betti number of the quotient graph $T/G$.
This is equal to the fundamental group of this graph, which is $1-|V(T/G)|+|E(T/G)|$.
The following lemma is standard, see for instance \cite{levitt2007gbs}, and says that, apart from in a single case, the Betti number is a group invariant.

\begin{lemma}\label{lem:bettiIndependence}
    Suppose $G$ is a GBS group not isomorphic to the Klein bottle group. Then the Betti number of a GBS system $(G,T)$ does not depend on the choice of $T$.
\end{lemma}

\p{The modulus of a GBS group}
A GBS group is \emph{non-elementary} if it is not isomorphic to $\mathbb{Z}$, $\mathbb{Z}^2$ or the Klein bottle group.

%\begin{definition}
    Given a non-elementary GBS group $G$, the \emph{modular homomorphism} $\Delta: G \rightarrow \mathbb{Q}^*$ is defined as follows. For any $g$, take $x$ acting elliptically in a GBS tree $T$. Then there exists $p, q \in \mathbb{Z}$ such that $gx^pg^{-1} = x^q$, and $\Delta(g) := \frac{p}{q}$. 
%\end{definition}

Note that all subgroups of elementary GBS groups are finitely generated of rank at most 2, and so the same is true of all fixed subgroups. As such, we will freely restrict to the non-elementary case whenever we need $\Delta$.

The map $\Delta$ does not depend on the choice of $T$ or $x$ \cite{kropholler1990note}. It is not hard to check that $\Delta$ is trivial on the kernel of the forgetful map $p_*$, and so $\Delta$ factors through $p_*$.

\subsection{Presentations and Fundamental Domains}
\label{sec:combinatorialFundamentalDomain}

In general, we prefer to work with groups acting on trees, avoiding the choices involved with writing a graph of groups or the corresponding presentation. When forced to fix a graph of groups, we will do so by choosing a fundamental domain in the Bass-Serre tree:
%
%
%
%\begin{definition}\label{def:combinatorialFundamentalDomain}
    Given a $G$-tree $T$, a \emph{fundamental domain} is a subtree $K$ meeting every orbit exactly once. If any point in the interior if an edge is in $K$, then every point in the interior of that edge is in $K$. 
%\end{definition}

\begin{notation}
    Given a GBS system $(G,T)$ with a fundamental domain $K$, we take the quotient graph of groups $T//G$ (where we suppress the choice of fundamental domain) to be $T/G$, where each (open) edge $e \in E(T/G)$ has an edge group $G_e$ equal to the stabiliser of the unique lift of $e$ in $K$, and each vertex has a vertex group $G_v$ equal to the stabiliser of the unique lift of $v$ in $K$.  
\end{notation}

\p{Presentations}
Given a GBS system and fundamental domain with a corresponding graph of groups $T//G$, we may build a presentation for $G$ by taking the fundamental group of the graph of groups $T//G$ in the usual way, by also choosing, for each $s \in V(T//G) \cup E(T//G)$, $x_s$ a generator of $G_s$; and for each $v \in \overline{K} \setminus K$ an element $t_v$ which sends the unique vertex of $K$ in the orbit of $v$ to $v$. We call these vertex generators and stable letters, respectively. In the presentation given by taking the fundamental group of $T//G$, the generators are exactly the vertex generators and stable letters.

We will call a presentation arising in such a way a \emph{presentation for $(G,T)$}.

\p{$\mathbb{Z}$-labelled graphs and $1$-freeness}
Sometimes, we will not care to remember the vertex groups of $T//G$ as subgroups of $G$. In these cases, since each vertex and edge group is isomorphic to $\mathbb{Z}$, the only remaining data is an inclusion of each edge group into the vertex groups on its endpoints. As such, we can write $T//G$ as a graph with integer labels at each endpoint of each edge.

%\begin{definition}
    A GBS system $(G,T)$ is \emph{1-free} if there is no edge in $T$ whose $G$-stabiliser is equal to the $G$-stabiliser of one of its endpoints.
%\end{definition}
This is equivalent to $T//G$, viewed as a labelled graph, having no $\pm 1$ labels.

\subsection{Automorphisms and compatible actions}
\label{sec:CompatibleAutomorphismsPreliminaries}

Let $G$ be any group and let $T$ be a $G$-tree. Then $\aut^T(G)$ denotes the maximal subgroup of $\aut(G)$ leaving $T$ invariant. More precisely, if we regard $\Omega: G \rightarrow \aut(T)$ as the action of $G$, we say $\phi$ is in $\aut^T(G)$ if the $G$-trees $(T, \Omega)$ and $(T, \Omega \circ \phi)$ are equivariantly isometric. This group arises as the stabiliser of $(T, \Omega)$ in the associated deformation space, in the sense of Forester \cite{forester2002deformation}.

For our purposes, a more convenient characterisation of $\aut^T(G)$ (shown to be equivalent in Lemma \ref{lem:equivalentCompatible}) uses compatible actions:
    Suppose $G$ acts on a tree $T$. If $\inn(G) \leq A \leq \aut(G)$ makes the diagram in Figure \ref{fig:CommDiagBassJiang} commute, then $A$ has a \emph{compatible action} on $T$.
The subgroup $\aut^T(G)$ is the maximal such subgroup.
Note that each such a subgroup $A\leq\aut^T(G)$ has an action on $T$ extending the action of $G$, justifying the ``compatible'' label.

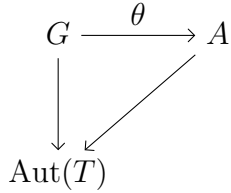
\begin{figure}[ht]
\begin{centering}
\begin{tikzpicture}[>=angle 90]
\matrix(a)[matrix of math nodes, row sep=3em, column sep=2.5em, text height=1.5ex, text depth=0.25ex] {G& A\\ \aut(T)\\};
\path[->](a-1-1) edge node[auto] {$\theta$}(a-1-2);
\path[->](a-1-2) edge (a-2-1);
\path[->](a-1-1) edge (a-2-1);
\end{tikzpicture}
\caption{\footnotesize The subgroup $\aut^T(G)$ is the maximal subgroup of $\aut(G)$ containing $\inn(G)$ making the diagram commute,  where $\theta$ is the canonical homomorphism with $\theta(G)=\inn(G)$.}\label{fig:CommDiagBassJiang}
\end{centering}
\end{figure}

We will always write $\cdot$ for the action of an automorphism, and write the action of $g$ by juxtaposition.

\p{Consequences of compatibility}
We will discuss several properties of compatible actions which shall be used throughout the paper, before proving that $\aut^T(G)$ is the unique maximal subgroup with a compatible action.

\begin{lemma}
\label{lem:Compatibility}
Suppose $A\leq\aut(G)$ has a compatible action on $T$, and let $\phi\in A$.
%where $\inn(G) \leq A \leq \aut(G)$ and $A$ makes the diagram in Figure \ref{fig:CommDiagBassJiang} commute.
Then for all $g\in G$ and all $x\in T$ we have $\phi\cdot gx=\phi(g)(\phi\cdot x)$.
\end{lemma}

\begin{proof}
We write $\phi_g$ for the inner automorphism corresponding to conjugation by $g$, so $\phi_g(h)=g^{-1}hg$ for all $h\in G$. Then $\phi\phi_g\phi^{-1}=\phi_{\phi(g)}$ as for all $h\in G$ we have:
\[
\phi\phi_g\phi^{-1}(h)=\phi(g^{-1}\phi^{-1}(h)g)=\phi(g^{-1})h\phi(g)=\phi_{\phi(g)}(h)
\]
We therefore have
\[
\phi\cdot gx = \phi\phi_g\cdot x = \phi\phi_g\phi^{-1}\phi \cdot x = \phi_{\phi(g)}\phi \cdot x = \phi(g)(\phi \cdot x)
\]
as required.
\end{proof}

Lemma \ref{lem:Compatibility} is essentially a notational tool, which we use without reference throughout the paper.

A particularly important special case of Lemma \ref{lem:Compatibility} will be Corollary \ref{cor:MinsetAction}. Here, and throughout, given $T$ a tree and an isometry $\gamma \in \mathrm{Iso}(T)$ (or possibly an element of a group acting on $T$, which we conflate with its action), we define \[T^\gamma := \{x \in T \mid d(x, \gamma x) = \ell(\gamma)\},\] that is $T^\gamma$ denotes the \emph{minset} of $\gamma$, which is always non-empty when $T$ is a tree. 

\begin{corollary}\label{cor:MinsetAction}
    Suppose $\phi$, $G$ and $T$ are as in \Cref{lem:Compatibility}. Then the action $G \curvearrowright T$ restricts to an action $\fix(\phi) \curvearrowright T^\phi$.
\end{corollary}
\begin{proof}
    Suppose $x \in T^\phi$ and $g \in \fix(\phi)$. Then $$d(gx, \phi \cdot gx) = d(gx, g(\phi \cdot x)) = d(x, \phi \cdot x) = \ell(\phi),$$ so in particular $gx \in T^\phi$.
\end{proof}

\p{Compatibility and $\aut^T(G)$}
With the tools of the previous results in hand, we are ready to prove that we can study all of $\aut^T(G)$ at once with one compatible action.

\begin{lemma}\label{lem:equivalentCompatible}
    Let $G$ be a group acting minimally on a tree $T$ which is not a line. Then $\aut^T(G)$ is the maximal subgroup of $\aut(G)$ making the diagram in Figure \ref{fig:CommDiagBassJiang} commute.
\end{lemma}
\begin{proof}
    Write $A$ for an arbitrary subgroup of $\aut(G)$ making the diagram commute.

    Suppose that $\phi \in A$. Then, by unwinding the definition and applying Lemma \ref{lem:Compatibility}, one immediately sees that $\phi \cdot$ (that is the isometry of $T$ coming from the action of $\phi$) is an equivariant isometry from $(T, \Omega)$ to $(T, \Omega \circ \phi)$.

    Conversely, suppose $\phi \in \aut^T(G)$. Then we show there is in fact a unique $f: T \rightarrow T$ an equivariant isometry from $(T, \Omega)$ to $(T, \Omega \circ \phi)$ (henceforth we will call such an $f$ \emph{$\phi$-equivariant} for brevity). To prove this claim, consider $g$ a distinct $\phi$-equivariant isometry, and note that $f^{-1}g$ is $id$-equivariant. Since $f^{-1}g$ is an automorphism of $T$ it leaves a subtree $T^{f^{-1}g}$ invariant, which is either its axis or the set of points it fixes. It follows by Corollary \ref{cor:MinsetAction} that the action of $G$ restricts to $T^{f^{-1}g}$ (since $\fix(id) = G$). By the assumption that the $T$ is a minimal $G$-tree, we see that $T^{f^{-1}g} = T$. It follows that $f^{-1}g$ doesn't have an axis, since $T$ is not a line, and so instead it fixes all of $T = T^{f^{-1}g}$. Hence, $f = g$.
    
    It is not hard to check that the composition of a $\phi$-equivariant and $\psi$-equivariant isometry is $\phi\psi$-equivariant, and likewise the inverse of a $\phi$-equivariant isometry is $\phi^{-1}$-equivariant. So we may build a compatible action of $\aut^T(G)$ on $T$ by sending each $\phi$ to the unique $\phi$-equivariant isometry.
\end{proof}

As compatibility defines $\aut^T(G)$, we will study fixed subgroups of automorphisms $\phi \in \aut^T(G)$ by letting $\phi$ act on $T$, and considering the action of $\fix(\phi) \curvearrowright T^\phi$. We are especially interested in this action when $\phi$ fixes a point on $T$, so $T^\phi$ is the set of points fixed by $\phi$.

\begin{lemma}
\label{lem:nothingFixes}
Suppose $\phi\in\aut^T(G)$ fixes a point $x\in T$. Then $\phi$ restricts to an automorphism on $G_x$.
\end{lemma}
\begin{proof}
Suppose $\phi$ fixes $x$, i.e. $\phi\cdot x=x$, and suppose $g\in G_x$.
Then we have
\[
x=\phi \cdot x=\phi \cdot gx=\phi(g)(\phi \cdot x)=\phi(g)x
\]
and so $\phi(g)\in G_x$.
Hence, $\phi$ acts as an endomorphism on $G_x$.
As $\phi\cdot x=x$ we also have $\phi^{-1}\cdot x=x$, and by an identical argument to the above we have that $\phi^{-1}$ also acts as an endomorphism on $G_x$. Hence, $\phi|_{G_x}$ has an inverse and so is an automorphism as required.
\end{proof}

\p{The sign of an automorphism}
In light of Lemma \ref{lem:nothingFixes}, the stabiliser of a point $x \in T^\phi$ under the action of $\fix(\phi)$ is $\fix(\phi) \cap G_x = \fix(\phi|_{G_x})$. It turns out that the behaviour of fixed subgroups will depend on how $\phi$ acts on $G_x$ in such cases.

More precisely,
    given $(G,T)$ a GBS system, we define a map $\sgn: \aut^T(G) \rightarrow \{\pm 1\}$, which we call the \emph{sign} of a compatible automorphism, as follows:
    if $\phi \in \aut^T(G)$ fixes a point in $T$, we say $\sgn(\phi) = 1$ if there is $x \in T^\phi$ such that $\phi$ restricts to the identity on $G_x$, and $\sgn(\phi) = -1$ otherwise;
    if $\phi$ does not fix a point on $T$, we say $\sgn(\phi)$ is undefined.

Notice that ``some'' in the previous definition could have been replaced by ``any'', since all stabilisers in a GBS tree are commensurable copies of $\mathbb{Z}$. Thus, if $\phi$ is not the identity on some $G_x$, then it acts by inversion on every $G_x$ (i.e. by the only non-trivial automorphism of $\mathbb{Z}$).

The final preliminary lemma uses $\sgn(\phi)$.
This lemma is applied in combination with Proposition \ref{prop:bassFg} to exhibit fixed subgroups which are not finitely generated.

\begin{lemma}\label{lem:sign1minimal}
    Suppose $(G,T)$ is a 1-free GBS system and $\phi \in \aut^T(G)$ has sign $1$. Then the action $\fix(\phi) \curvearrowright T^\phi$ is minimal.
\end{lemma}
\begin{proof}
    Suppose, seeking a contradiction, that the action of $\fix(\phi)$ on $T'$ is not minimal.
    Then there exists a proper $\fix(\phi)$ invariant subtree $T'' \subset T'$.

    Take vertices $u, v \in T'$, such that $u \in T''$ but $v \notin T''$. Write $\gamma$ for the path from $u$ to $v$ and $e$ for the unique edge of $\gamma$ with $v$ as an endpoint. Since $(G,T)$ is 1-free, we may take $g \in G_v \setminus G_e$. Since $\sgn(\phi) = 1$, $g \in \fix(\phi)$, but $g$ sends $T''$ to a disjoint subtree, contradicting that $T''$ is $\fix(\phi)$ invariant.
\end{proof}

\section{Fixed subgroups of compatible automorphisms}

In this section, given a GBS system $(G,T)$ and an automorphism $\phi \in \aut^T(G)$, we investigate when $\fix(\phi)$ is finitely generated. We will prove Theorems \ref{thmintro:characterisingFg} and \ref{thmintro:allFiniteOrder}.

\subsection{Automorphisms not fixing a point of $T$}
We firstly consider what happens when $\phi$ does not fix any point of $T$.

We start with the following proposition, which is a strengthening of Sykiotis' \cite[Proposition 3.1]{sykiotis2002GraphOfGroups}.

\begin{proposition}
\label{prop:Sykiotis}
Let $G$ be a group acting on a tree $T$, and suppose $\phi\in\aut^T(G)$.
Suppose $\phi$ does not fix any point of $T$.
%and that there exists some $g\in\fix(\phi)$ which does not fix any point of $T$.
Then there exists an edge $e\in ET$ such that one of the following holds:
\begin{enumerate}
    \item $\fix(\phi)\leq G_e$.
    \item\label{Sykiotis:2} $\fix(\phi)=(\fix(\phi)\cap G_e)\rtimes\mathbb{Z}$.
\end{enumerate}
\end{proposition}

\begin{proof}
%Note as in Lemma \ref{lem:translation}, that there is a unique path $p$ along which $\phi$ acts by translation.
Consider the axis $T^\phi$, on which $\fix(\phi)$ acts by Corollary \ref{cor:MinsetAction}.

If $\fix(\phi)$ fixes $T^\phi$ then, as $T^\phi$ is infinite and so contains an edge, $\fix(\phi)$ fixes an edge $e$ of $T$, and so $\fix(\phi)\leq G_e$ as required.

Suppose therefore that $\fix(\phi)$ does not fix $T^\phi$. We claim that the action is by translations. Indeed take $g \in \fix(\phi)$ fixing a point $x \in T$, then $$\phi \cdot x = \phi \cdot gx = g (\phi \cdot x),$$ so $g$ also fixes a distinct point $\phi \cdot x$ and thus necessarily fixes the whole axis.

Translation length along $T^\phi$ induces a surjection $\psi:\fix(\phi)\twoheadrightarrow\mathbb{Z}$, where $\ker(\psi)$ fixes this path and we have $\ker(\psi) = \fix(\phi)\cap G_{T^\phi}$.
As $\fix(\phi)$ acts by translation on $T^\phi$, we have that $\fix(\phi)\cap G_{T^\phi}=\fix(\phi)\cap G_e$ for all edges $e$ on the path ${T^\phi}$, and so there exists an edge $e\in ET$ such that $\ker(\psi) = \fix(\phi)\cap G_e$.
The result now follows as every map to $\mathbb{Z}$ splits, so $\fix(\psi)=\ker(\psi)\rtimes\mathbb{Z}=(\fix(\phi)\cap G_e)\rtimes\mathbb{Z}$ as required.
\end{proof}

We may now apply the above to GBS groups.
An equivalent formulation of this theorem is that $\fix(\phi)$ embeds into the Klein bottle group $\mathbb{Z}\rtimes\mathbb{Z}$.

\begin{theorem}
\label{thm:phi_hyperbolic}
Let $(G,T)$ be a GBS system, and let $\phi\in\aut^T(G)$.
Suppose $\phi$ does not fix any point of $T$.
Then one of the following holds:
\begin{enumerate}
    \item $\fix(\phi)$ is trivial.
    \item $\fix(\phi)$ is infinite cyclic.
    \item $\fix(\phi)\cong \mathbb{Z}\times\mathbb{Z}$
    \item $\fix(\phi)$ is isomorphic to the Klein bottle group (i.e. to the non-trivial semidirect product $\mathbb{Z}\rtimes\mathbb{Z}$).
\end{enumerate}
In particular, $\rk(\fix(\phi))\leq2$.
\end{theorem}

\begin{proof}
    As each edge stabiliser $G_e$, $e\in ET$, embeds into a vertex stabliser $G_v$, $v=\iota(e)$, and so is infinite cyclic, the result follows from \Cref{prop:Sykiotis}.
\end{proof}

\Cref{prop:Sykiotis} is completely general, and indeed the proof of \Cref{thm:phi_hyperbolic} only applies that all edge groups $G_e$ are infinite cyclic; this is the case for GBS groups, but also for example in JSJ decompositions for torsion-free hyperbolic groups.
If we additionally assume that all edge groups are either both infinite cyclic and central, or are trivial, which for example allows free products, then the action in the semidirect product of \Cref{prop:Sykiotis}.\ref{Sykiotis:2} must be trivial and so $\fix(\phi)$ is either $\mathbb{Z}$ or $\mathbb{Z}^2$.

\subsection{Automorphisms with sign $-1$}
Throughout the remainder of this section, let $(G,T)$ be an arbitrary GBS system, and $\phi \in \aut^T(G)$. 

We now deal with the case where $\sgn(\phi)=-1$, which also leads to well-behaved fixed point subgroups. We begin by realising $\fix(\phi)$ as the fundamental group of a graph, and hence free. We then give a finite bound for the rank of $\fix(\phi)$ in terms of the graph $T/G$.

\begin{lemma}\label{lem:negSgnFree}
Suppose $\sgn(\phi) = -1$.
Then $\fix(\phi)$ is isomorphic to $\pi_1(T^{\phi}/\fix(\phi))$.
\end{lemma}

\begin{proof}
By Corollary \ref{cor:MinsetAction}, $\fix(\phi)$ acts on the tree $T^\phi$ with stabilisers $\fix(\phi)\cap G_x$.
By assumption on the sign, each stabiliser is trivial.
The action is therefore free, and the result follows.
\end{proof}

We now prove that the graph $T^\phi / \fix(\phi)$ is finite.

\begin{lemma}\label{lem:negSgnTreeStructure}
Suppose $\sgn(\phi) = -1$.
Then the graph $T^{\phi}/\fix(\phi)$ has at most twice as many vertices and twice as many edges as the graph $T/G$.
\end{lemma}

\begin{proof}
Take a point $x \in T^\phi$, and denote a generator of $G_x$ by $a$.
Now take $gx \in T^\phi$ in the $G$ orbit of $x$ and assume it is not in the $\fix(\phi)$ orbit of $x$. Then,
\[
gx = \phi \cdot gx = \phi(g)(\phi\cdot x) = \phi(g)x,
\]
where the first and last equalities are by assumption that $x, gx \in T^\phi$ and the second is by \Cref{lem:Compatibility}. In particular, $g^{-1} \phi(g) \in G_x$. We write $a$ for the generator of $G_x$, and so $g^{-1}\phi(g) = a^k$ for some $k\in\mathbb{Z}$. 

In fact, we can find $g'$ such that $g'x = gx$ and $g'^{-1}\phi(g') = a$. To see this, suppose $g^{-1}\phi(g) = a^{2n+r}$, where $r \in \{0,1\}$. Then set $g' = ga^n$. Clearly $g'x = gx$, and
\[
g'^{-1}\phi(g') = (ga^n)^{-1}\phi(ga^n) = a^{-n}g^{-1}\phi(g)a^{-n} = a^r,
\]
where the second equality uses the fact that $\phi(a) = a^{-1}$ since $\sgn(\phi) = -1$. It must be that $r = 1$, or we have $g'^{-1}\phi(g') = 1$, that is $g' \in \fix(\phi)$, contradicting the fact that $x$ and $g'x$ are in different $\fix(\phi)$ orbits. From here we will rename $g'$ to $g$ so $g^{-1}\phi(g) = a$.

Now take two points $gx, hx \in T^\phi$ in the $G$ orbit of $x$ but not the $\fix(\phi)$ orbit of $x$. By the reasoning above, we may assume $g^{-1}\phi(g) = a = h^{-1}\phi(h)$, and so $hg^{-1} \in \fix(\phi)$. This means that $gx$ and $hx$ are in the same $\fix(\phi)$ orbit. Therefore there are at most two $\fix(\phi)$ orbits of points in $T^\phi$ for each $G$ orbit of point in $T$, so $T^\phi / \fix(\phi)$ is finite.
\end{proof}

We therefore have the following, using the fact that if $\Gamma$ is a graph, then $\rk(\pi_1(\Gamma)) = 1-|V\Gamma|+|E\Gamma|$.
%the Euler characteristic of a graph $\hat{\Gamma}$, defined as
%$\chi(\hat{\Gamma}) := 1-|V\hat{\Gamma}|+|E\hat{\Gamma}|$.
%Note that $\pi_1(\hat{\Gamma})$ has rank precisely $\chi(\hat{\Gamma})$, and that $2|E\hat{\Gamma}| = 2\chi(\hat{\Gamma})+2|V\hat{\Gamma}|-2$, which gives a different form for the identity in \Cref{prop:negSgnFg}.

\begin{proposition}\label{prop:negSgnFg}
If $\sgn(\phi) = -1$, then $\fix(\phi)$ is a finitely generated free group, of rank at most $2|E(T/G)|$.
\end{proposition}

\begin{proof}
Write $\Gamma = T/G$. By \Cref{lem:negSgnFree,lem:negSgnTreeStructure}, $\fix(\phi)\cong \pi_1(\Gamma')$ for some graph $\Gamma'$ such that $|V\Gamma'|\leq2|V\Gamma|$ and $|E\Gamma'|\leq2|E\Gamma|$.
We therefore have the following:
\begin{align*}
    \rk(\fix(\phi))
    %&=\chi(\Gamma')\\
    &=1-|V\Gamma'|+|E\Gamma'|\\
    &\leq |E\Gamma'|\\
    &\leq 2|E\Gamma|
\end{align*}
The result follows immediately.
\end{proof}

\subsection{Automorphisms with sign $1$} In this final subsection, we study automorphisms with sign $1$. In contrast to the other two cases, non-finitely generated fixed subgroups can occur. We find some sufficient conditions for fixed subgroups to be finitely generated.

As before, our main tool is Corollary \ref{cor:MinsetAction}.

\begin{lemma}\label{lem:orbitsInTphi}
    Suppose $\sgn(\phi) = 1$. Let $e_1 = \{v, u_1\}$ and $e_2 = \{v, u_2\}$ be two edges of $T$, which share an endpoint and are in the same $G$-orbit. Suppose that $e_1 \in E(T^\phi)$. Then $e_2 \in E(T^\phi)$, and they are in the same $\fix(\phi)$-orbit.
\end{lemma}
\begin{proof}
    Since $e_1$ and $e_2$ share an endpoint, there is $g \in G_v$ such that $ge_1 = e_2$. Now, we notice that $G_v \leq \fix(\phi)$, by assumption that $\sgn(\phi) = 1$. The conclusion follows.
\end{proof}

\begin{lemma}\label{lem:tphiInheretsLabels}
    Suppose $\sgn(\phi) = 1$. Then the injection $T^{\phi}\hookrightarrow T$ induces an immersion of labelled graphs $i:T^{\phi}//\fix(T^{\phi})\looparrowright T//G$.
    %$T^\phi//\stab(T^\phi)$ inherets labels]
\end{lemma}
\begin{proof} 
    The inclusion of subtrees induces a local injection $T^\phi//\fix(T^\phi) \looparrowright T//G$ since edges of $T^\phi$ sharing an endpoint in distinct $\fix(T^\phi)$-orbits are not in the same $G$-orbit by Lemma \ref{lem:orbitsInTphi}.
    
    The other point is that the immersion $i: T^\phi // \fix(T^\phi) \looparrowright T//G$ respects the labels. To see this, take $v \in T^\phi$, and $e_i = \{v, u_i\}$ for $i \in \{1 \dots n\}$ to be $n$ distinct edges in $T$ in the same $G$-orbit. Suppose also there are no other edges in this orbit with $v$ as an endpoint (so in $T//G$, the label on the image of $e_i$ near the image of $v$ is $n$). Suppose that $e_1 \subseteq T^\phi$. Then, by Lemma \ref{lem:orbitsInTphi}, for $i \in \{2 \dots n\}$, $e_i \in T^\phi$ and there is $g_i \in \fix(T^\phi)$ such that $g_ie_1 = e_i$. It can't be the case that $\fix(T^\phi) \curvearrowright T^\phi$ has new edges in this orbit with $v$ as an endpoint (as $\fix(T^\phi) \leq G$ and $T^\phi \subseteq T$), so we have shown the label in $T^\phi//\fix(T^\phi)$ on the edge which is the image of $e_1$ near the image of $v$ is exactly $n$, completing the proof.
\end{proof}

\begin{theorem}
\label{thm:contractible}
If $\sgn(\phi) = 1$ and $p(T^\phi)$ is a tree, then $\fix(\phi)$ is finitely generated, of rank at most $|Vp(T^\phi)|$.
\end{theorem}
\begin{proof}
    Suppose $p(T^\phi)$ is a tree and take $g \in \stab(T^\phi)$.
    %Take $v \in T^\phi$ a basepoint (as in Proposition \ref{posSgn}, we can take the basedpoint in $T^\phi$ without loss of generality).
    %Consider the map $p_*:G\to\pi_1(T/G, p(v))$, and as usual we may assume that $v\in VT^{\phi}$.
    %Since $g \in \stab(T^\phi)$, $gv \in T^\phi$ also, and therefore so is $[v, gv]$. This means $p([v, gv]) \subseteq p(T^\phi)$. Since $p(T^\phi)$ is a tree, $p([v,gv])$ is null-homotopic, which is to say $p_*(g)$ is trivial. So $\stab(T^\phi) \leq \ker(p_*|_{\stab(T^\phi)})$.
    %By the chain of inclusions, $$\ker(p_*|_{\stab(T^\phi)}) \leq \fix(\phi) \leq \stab(T^\phi)$$
    %which are given by \Cref{cor:posSgn} and Corollary \ref{cor:MinsetAction},
    %this shows that $\fix(\phi) = \stab(T^\phi)$.

    %It therefore only remains to show that $\stab(T^\phi)$ is finitely generated, with the required bound on its rank.
    Consider the graph of groups $T^\phi // \fix(T^\phi)$. By \Cref{lem:tphiInheretsLabels}, the underlying graph $T^\phi / \fix(T^\phi)$ locally injects into $T/G$. In fact its image must be in $p(T^\phi)$, since the local injection is induced by the inclusion $T^\phi \hookrightarrow T$, but given $p(T^\phi)$ is a tree this local injection is an injection, so $T^\phi / \fix(T^\phi)$ is a finite tree. Moreover all of the vertex $T^\phi // \fix(T^\phi)$ are infinite cyclic.
    Since $\fix(T^\phi)$ is the fundamental group of a finite tree of cyclic groups, it is finitely generated by Proposition \ref{prop:bassFg}, of rank at most $|Vp(T^\phi)|$ as required.
\end{proof}

\begin{corollary}
\label{corol:contractible_combined}
Let $G$ be a GBS group, and let $\phi\in\aut^T(G)$.
If $T/G$ is a tree then $\fix(\phi)$ is finitely generated, of rank at most $\max(1, 2|E(T/G)|)$.
\end{corollary}

\begin{proof}
    Write $\Gamma = T/G$.

    The result is clear when $\Gamma$ consists of a single point, as $G$ is cyclic so $\rk(\fix(\phi))\leq 1=\max(1, 2|E\Gamma|)$.
    Suppose therefore that $\Gamma$ contains an edge; here $2|E\Gamma|=\max(1, 2|E\Gamma|)$.

    If $\sgn(\phi)$ is undefined or $\sgn(\phi) = -1$, then we can apply Theorem \ref{thm:phi_hyperbolic} or Proposition \ref{prop:negSgnFg} respectively, and the result holds.
    Otherwise, $\phi$ satisfies the assumptions of \Cref{thm:contractible} since $p(T^\phi)$ is a subtree of $\Gamma$, and so $\rk(\fix(\phi))\leq |V p(T^\phi)| \leq |V\Gamma| \leq 2|E\Gamma|$ as required.
\end{proof}

We now turn to the case where $(G,T)$ has Betti number one. In this case, the behaviour depends on the generator of $\Delta(G) \leq \mathbb{Q}^*$ (notice that, since $\Delta$ factors through $p_*$, in this case $\Delta(G)$ is an infinite cyclic subgroup so it makes sense to talk about it being generated by a single element).

\begin{proposition}\label{prop:BettiNumber1fgFix}
     Suppose $(G,T)$ is a non-elementary GBS system of Betti number 1, such that  $\Delta(G) \leq \mathbb{Q}^*$ is not generated by an integer except possibly $-1$. Then for all $\phi \in \aut^T(G)$, $\fix(\phi)$ is finitely generated.

     Moreover, if $\Delta(G) = \{1, -1\}$, then $\rk(\fix(\phi)) \leq 2|V(T/G)|+1$.
\end{proposition}
\begin{proof}
    If $\phi$ does not fix a point on $T$ or if $\sgn(\phi) = -1$, then the result follows from Theorem \ref{thm:phi_hyperbolic} and Proposition \ref{prop:negSgnFg} respectively.
    This includes the ``moreover'' part of the statement, since the bound on rank from Proposition \ref{prop:negSgnFg} is $2|E(T/G)| \leq 2|V(T/G)|+1$ since $E|(T/G)| = |V(T/G)|$ when $T/G$ has Betti number 1.
    Now we turn to the case where $\sgn(\phi) = 1$.

    By Proposition \ref{prop:bassFg} it is sufficient to show that the graph $T^\phi/\fix(\phi)$ is finite, since it can be turned into a graph of groups for $\fix(\phi)$ with all vertex groups being $\mathbb{Z}$. Suppose for contradiction that $T^\phi/\fix(\phi)$ is infinite.

    \begin{claim}\label{claim:powerRay}
        There exists $v \in T^\phi$ and $g \in G$ such that $p_*(g)$ generates $\pi_1(T/G)$, such that $g^kv \in T^\phi$ for all $k \geq 0$. Furthermore, each vertex $g^kv$ is in a different $\fix(\phi)$ orbit.
    \end{claim}
    \begin{proof}[Proof of Claim.]
    
    We write $p: T \rightarrow T/G$ and $q: T^\phi \rightarrow T^\phi / \fix(\phi)$ for the quotient maps, and $i: T^\phi \rightarrow T$ for the inclusion.
    
     Since $T^\phi/\fix(\phi)$ is infinite, it in particular contains an infinite ray $\alpha : \mathbb{R}_+ \rightarrow T^\phi/\fix(\phi)$. Without loss of generality assume $\alpha$ has a reduced image in $T/G$ under the immersion $\mu: T^\phi/\fix(\phi) \looparrowright T/G$ from Lemma \ref{lem:tphiInheretsLabels}; since $T/G$ has Betti number $1$, the image of $\alpha$ under $\mu$ must wrap around the unique reduced loop infinitely many times. 
    
    Take $v$ to be a lift to $T^\phi$ of the initial vertex of $\alpha$ in $T^\phi/\fix(\phi)$, and $gv$ to be the next vertex in the same $G$-orbit on a lift of $\alpha$ to $T^\phi$ based at $v$ (notice that the lift is not unique since the action on $T^\phi$ is not free). Since $\mu\alpha$ follows the unique reduced loop in $T/G$, we see that $p_*(g)$ generates $\pi_1(T/G)$.
    
    Now consider the ray $\beta$ which is the unique reduced ray (in $T$) passing through $g^kv$ for all $k \geq 0$. We will now show that $\beta$ is in fact a ray in $T^\phi$.

    We first observe that $p\beta = \mu\alpha$, by choice of $g$ and $v$. Consider $\hat{\alpha}$, a lift of $\alpha$ to $T^\phi$ which coincides with $\beta$ maximally; such a lift exists by Zorn's lemma, noting that the set of lifts of $\alpha$ coinciding with $\beta$ non-trivially is non-empty since it contains the lift used to define $g$. Suppose for contradiction that $\hat{\alpha}$ and $\beta$ diverge. This necessarily happens at a vertex which we call $u$. Now, the next edges in each $\hat{\alpha}$ and $\beta$, $e_{\hat{\alpha}}$ and $e_{\beta}$, have $u$ as an endpoint, and are in the same $G_u$-orbit, since the images of $\hat{\alpha}$ and $\beta$ in $T/G$ agree. But $G_u \leq \fix(T^\phi)$ since $u \in T^\phi$ and $\sgn(\phi) = 1$, so in fact, since $\fix(\phi)$ leaves $T^\phi$ invariant, $e_{\hat{\alpha}}$ is in $T^\phi$ since $e_\beta$ is. Since the edges are in the same $\fix(\phi)$ orbit, we could have extended $\hat{\alpha}$ along $e_\beta$ instead of $e_{\hat{\alpha}}$, contradicting the maximality of $\hat{\alpha}$. Hence, we have shown that $\hat{\alpha}$ and $\beta$ coincide, so $g^kv \in T^\phi$ for all $k \geq 0$.

    Finally we have to check that, for all $k$, $g^kv$ are in different $\fix(\phi)$ orbits. We observe that $\beta$ is a ray in $T^\phi$, so we may write $\mu q\beta = pi\beta = \mu \alpha$. Since $\mu$ is a local injection and $\alpha$ and $q\beta$ share an initial segment, they coincide, but this exactly means each $g^kv$ is in a different $\fix(\phi)$ orbit.
    \end{proof}

    We will now use the points $g^kv \in T^\phi$ to produce a contradiction with the assumption on $\Delta(G)$. Observe that $$g^kv = \phi \cdot g^kv = \phi(g^k) v,$$ so there are integers $l_k$ such that $\phi(g^k) = g^kx^{l_k}$ for all $k$, where $x$ generates $G_v$. Write $l = l_1$ for the sake of notation.

    \begin{claim}\label{claim:geometricSeries}
        For $k \geq 1$, $l_k = \sum_{i = 0}^{k-1} \Delta(g)^il$
    \end{claim}
    \begin{proof}[Proof of Claim.]
        We prove the claim by induction. The base case is clear. For the inductive step, observe that $g^{k+1}x^{l_{k+1}} = \phi(g^{k+1}) = g^kx^{l_k}gx^l$, and in turn $gx^{l_{k+1}-l}g^{-1} = x^{l_k}$, so $l_k\Delta(g) = l_{k+1} - l$. The result follows.
    \end{proof}

    Since $\pi_1(T/G)$ is generated by $p_*(g)$ and $\Delta$ factors through $p_*$, we see that $\Delta(g)$ generates $\Delta(G)$. So by assumption, either $\Delta(g) \notin \mathbb{Z}$ or $\Delta(g) = -1$.

    In the case that $\Delta(g) = -1$, $l_2 = 0$ so $\phi(g^2) = g^2$. This contradicts that, by Claim \ref{claim:powerRay}, each $g^kv$ is in a different $\fix(\phi)$ orbit.
    In fact, in this case, we have shown more: the graph $T^\phi/\fix(\phi)$ can have at most twice as many vertices as $T/G$. However $T^\phi/\fix(\phi) \looparrowright T/G$, so the Betti number of $T^\phi/\fix(\phi)$ is 1. The bound on the rank of $\fix(\phi)$ follows.

    In the case that $\Delta(g) \in \mathbb{Q}\setminus\mathbb{Z}$, we have that $l_k = \sum_{i = 0}^{k-1} \Delta(g)^il$, but it is clear that this is not an integer for sufficiently large $k$, which again is a contradiction. 
    
\end{proof}

 We conclude this section by considering finite order automorphisms, proving Theorem \ref{thmintro:allFiniteOrder}, and then summarising sufficient conditions we have established for $\phi \in \aut^T(G)$ to have a finitely generated fixed subgroup.

\begin{lemma}\label{lem:finiteorder}
    Let $(G,T)$ be a GBS system, and take $\phi \in \aut^T(G)$ of finite order. Then $\fix(\phi)$ is finitely generated.
\end{lemma}
\begin{proof}
    As usual, in light of Corollary \ref{cor:MinsetAction}, we consider $\fix(\phi) \curvearrowright T^\phi$, and restrict to the case $\sgn(\phi) = 1$ by Theorem \ref{thm:phi_hyperbolic} and Proposition \ref{prop:negSgnFg}. 
    
    Suppose $v, gv \in T^\phi$ are in the same $G$-orbit. Then
    \[
    gv = \phi \cdot gv = \phi(g) (\phi \cdot v) = \phi(g)v
    \]
    and it follows that $\phi(g) = gx$ for some $x \in G_v$. However, since $\sgn(\phi) = 1$, we observe that, for all $k \in \mathbb{N}$, $\phi^k(g) = gx^k$. Since $G$ is torsion-free (it acts on a tree with torsion-free stabilisers) and $\phi$ is of finite order, it must be that $x = 1$, and so $v$ and $gv$ are in the same $\fix(\phi)$-orbit. So $T^\phi/\fix(\phi)$ is finite, and the result follows by Proposition \ref{prop:bassFg}.
\end{proof}

\begin{corollary}\label{cor:allFiniteOrder}
    Suppose $G$ is a GBS group and $\phi \in \aut(G)$ is of finite order. Then $\fix(\phi)$ is finitely generated.
\end{corollary}
\begin{proof}
    Since $\phi$ is of finite order, $\langle \phi \rangle \leq \aut(G)$ is a finite cyclic, and thus solvable, subgroup. By \cite[Corollary 8.4]{guirardel2007deformation}, $\phi$ fixes a point on the deformation space of GBS trees. It follows that $\phi$ leaves a GBS tree invariant, so Lemma \ref{lem:finiteorder} applies.
\end{proof}

\begin{theorem}\label{thm:sufficientFg}
    Suppose $(G,T)$ is a GBS system, and $\phi \in \aut^T(G)$. Then $\fix(\phi)$ is finitely generated if one of the following occur:

    \begin{enumerate}
        \item $\phi$ doesn't fix a point in $T$,
        \item $\phi$ fixes a point and acts non-trivially on the stabiliser,
        \item The image of $T^\phi$ in $T/G$ is a tree,
        \item $\phi$ is of finite order.
    \end{enumerate}
\end{theorem}
\begin{proof}
    The theorem follows immediately by combining Theorem \ref{thm:phi_hyperbolic}, Proposition \ref{prop:negSgnFg}, Theorem \ref{thm:contractible} and Lemma \ref{lem:finiteorder}.
\end{proof}

\subsection{Non-finitely generated fixed subgroups}
\label{sec:nonFG_examples}
In this subsection, we provide sufficient conditions on $1$-free GBS systems for non-finitely generated fixed subgroups to exist.

\begin{lemma}\label{lem:integralModInfiniteGen}
    Suppose $(G,T)$ is a non-elementary 1-free GBS system with a presentation containing a stable letter $t$ such that $\Delta(t) \in \mathbb{Q}^* \setminus \{-1\}$. 
    
    Then if $\Delta(t) \in \mathbb{Z}$, there is $\phi \in \aut^T(G)$ such that $\fix(\phi)$ is not finitely generated.

    Otherwise, if $\Delta(t) \notin \mathbb{Z}$ and $\Delta(t^{-1}) \notin \mathbb{Z}$, and furthermore $(G,T)$ is of Betti number $1$, then there are fixed subgroups of arbitrarily large finite rank.
\end{lemma}
\begin{proof}
    Fix the presentation for $G$ coming from a fundamental domain (see \Cref{sec:combinatorialFundamentalDomain}) with $t$ as a stable letter, which exists by hypothesis. Let $x$ be an arbitrary vertex generator in this presentation, with $v$ the corresponding vertex in the fundamental domain. We may choose a suitably large $p \neq 0$ such that $x^pt = tx^{p\Delta(t)}$. Now, for any $N \in \mathbb{Z}^*$, define $\phi_N$ such that $\phi_N(t) = tx^{pN}$ and $\phi_N$ acts trivially on all other generators. This is a twist automorphism, so is in $\aut^T(G)$ by work of Bass and Jiang \cite{bass1996automorphism} (see also \cite{levitt2005automorphisms}).

    Now define $l_i = \sum_{j = 0}^{i-1} \Delta(t)^j$, and $l_0 = 0$. We claim that, for $k \geq 0$, if $N\l_i$ is an integer for all $i < k$ then $\phi_N(t^k) = t^kx^{pNl_k}$. This is trivial for $k=0$, and we proceed inductively (under the assumption that $Nl_i$ is an integer for $i < k+1$), observing that $$\phi_N(t^{k+1}) = \phi_N(t^k)\phi_N(t) = t^kx^{pNl_k}tx^p = t^{k+1}x^{pN(l_k\Delta(t)+1)} = t^{k+1}x^{pNl_{k+1}},$$ where the penultimate equality uses that $Nl_k$ is an integer, and the final equality is by definition of $l_{i}$. Hence, we may conclude that for all $k \geq 0$ such that $Nl_i$ is an integer for $i < k$, $$\phi_N \cdot t^kv = \phi_N(t^k)(\phi_N \cdot v) = t^kv.$$ For convenience, from take $M_N$ to be maximal such that for all $0 \leq i < M_N$, $Nl_i$ is an integer (we potentially allow $M_N = \infty$). 

    Suppose that for $M_N > k_1, k_2 \geq 0$, $gt^{k_1}v = t^{k_2}v$ where $g \in \fix(\phi_N)$. Since $G_v$ is generated by $x$ it follows that, for some $i \in \mathbb{Z}$, $t^{-k_1}gt^{k_2} = x^i$, so $g = t^{k_1}x^it^{-k_2}$. Using that $$t^{k_1}x^it^{-k_2} = \phi_N(t^{k_1}x^it^{-k_2}),$$ and the previous claim, one directly computes that $l_{k_1} = l_{k_2}$. Since $\Delta(t) \neq -1$, this implies that $k_1 = k_2$. It follows that the vertices $t^kv$ for $M_N > k \geq 0$ are each in different $\fix(\phi_N)$ orbits of $T^\phi_N$.

    Suppose now that $\Delta(t) \in \mathbb{Z} \setminus \{-1\}$. Then $M_N = \infty$ (regardless of $N$), so $T^\phi_N$ contains infinitely many $\fix(\phi_N)$ orbits, and $\fix(\phi_N)$ is not finitely generated by Proposition \ref{prop:bassFg} and Lemma \ref{lem:sign1minimal} as required. 

    Suppose now that $\Delta(t), \Delta(t^{-1}) \notin \mathbb{Z}$ and $T/G$ has Betti number $1$. Now $\Delta(t) = \frac{p}{q} \notin \mathbb{Z}$ for $p, q \in \mathbb{Z}^*$, which we take to be co-prime. It must be that neither $p$ nor $q$ is equal to 1.
    
    By Proposition \ref{prop:BettiNumber1fgFix}, all fixed subgroups of automorphisms in $\aut^T(G)$ are finitely generated in this case. Since by Lemma \ref{lem:sign1minimal} the action $\fix(\phi_N) \curvearrowright T^{\phi_N}$ is minimal, the quotient is finite, so $(\fix(\phi_N), T^{\phi_N})$ is a GBS system. We can now apply \cite[Theorem 1.1]{levitt2015rank} on the rank of generalised Baumslag-Solitar groups. Take $q'$ an arbitrary prime factor of $q$. In the language of that paper, we aim to show that each of the vertices $t^kv$ for $M_N > k \geq 0$ is in a different $q'$-plateau, so the rank is at least $M_N$. Then by taking $N$ to be a large power of $q$, we may make $M_N$ arbitrarily large, completing the proof.
    
    We observe that, since $T^{\phi_N}/\fix(\phi_N) \looparrowright T/G$ and $T/G$ has Betti number $1$, any path between the images of vertices of $t^{k_1}v$ and $t^{k_2}v$ ($M_n > k_1, k_2 \geq 0$) in $T^\phi/\fix(\phi)$ is contained in the image of the geodesic in $T^\phi$ given by $[t^{k_1}v, t^{k_2}v]$. It follows that, considering $T^\phi//\fix(\phi)$ as a labelled graph, there is a label divisible by $q'$ on the end of an edge between the image of $t^{k_1}v$ and $t^{k_2}v$. One easily checks this means that these vertices are in different $q'$ plateaus.

\end{proof}

\begin{lemma}\label{lem:Betti2InfiniteGen}
    Suppose $(G,T)$ is a 1-free GBS system of Betti number at least 2. Then there is $\phi \in \aut^T(G)$ such that $\fix(\phi)$ is not finitely generated.
\end{lemma}
\begin{proof}
    Fix an arbitrary presentation for $(G,T)$ coming from a fundamental domain. Take $s, t$ to be two distinct stable letters (this is possible due to the Betti number). If $\Delta(s) = 1$, then by Lemma \ref{lem:integralModInfiniteGen} we are done, so suppose this is not the case.

    Fix $x$ to be an arbitrary vertex generator in the chosen presentation, and choose $v$ to be the corresponding vertex in the fundamental domain.

    Now, choose a $p \in \mathbb{Z} \setminus \{0\}$ such that the following conditions hold: \begin{enumerate}
        \item $x^ps = sx^{p\Delta(s)}$;
        \item $p\frac{(\Delta(s)-1)}{\Delta(s)\Delta(t)}$ is an integer;
        \item $x^{p(\Delta(s)-1)}t^{-1}s^{-1} = t^{-1}s^{-1}x^{p\frac{(\Delta(s)-1)}{\Delta(s)\Delta(t)}}$;
        \item $x^{p\frac{(\Delta(s)-1)}{\Delta(s)\Delta(t)}} tst^{-1}s^{-1} = tst^{-1}s^{-1}x^{p\frac{(\Delta(s)-1)}{\Delta(s)\Delta(t)}}$.
    \end{enumerate} This can always be arranged by taking $p$ as a product of suitably large integers making each condition work. For the final point, we are using the fact that $\Delta(tst^{-1}s^{-1}) = 1$.

    Now take $\phi$ to be the automorphism $\phi: t \mapsto tx^p$ which fixes all other generators. Again this is a twist automorphism.  Conditions (1) and (3) ensure that $\phi[t,s] = [t,s]x^{p\frac{(\Delta(s)-1)}{\Delta(s)\Delta(t)}}$. For simplicity write $\gamma = p\frac{(\Delta(s)-1)}{\Delta(s)\Delta(t)}$ and note that $\gamma \neq 0$ since we assumed $\Delta(s)$ was not 1. 

    We now compute, for any $k \in \mathbb{Z}$, \begin{align*}
        \phi \cdot [t,s]^kv &= \phi[t,s]^kv\\
        &= ([t,s]x^{p\frac{(\Delta(s)-1)}{\Delta(s)\Delta(t)}})^k(\phi \cdot v)\\
        &= [t,s]^kx^{kp\frac{(\Delta(s)-1)}{\Delta(s)\Delta(t)}}v\\
        &= [t,s]^kv.\\
    \end{align*}
    
    By exactly the argument in Lemma \ref{lem:integralModInfiniteGen}, each of these vertices is in a different $\fix(\phi)$ orbit, so $T^\phi$ contains infinitely many $\fix(\phi)$ orbits, and $\fix(\phi)$ is not finitely generated by Proposition \ref{prop:bassFg} and Lemma \ref{lem:sign1minimal} as required. 
\end{proof}

We are now ready to collect together the results from this section to complete the proofs of our main theorems.

\begin{theorem}[\Cref{thmintro:characterisingFg}]\label{thm:characterisingFg}
        Suppose $(G,T)$ is a $1$-free, non-elementary GBS system.
    Then $\fix(\phi)$ is finitely generated for all $\phi \in \aut^T(G)$ if and only if one of the following occurs:

    \begin{enumerate}
        %\item $T/G$ has Betti number 2,
        \item $\beta(G)=0$, or
        %\item $T/G$ has Betti number 1 and $\Delta(G)$ is generated by an integer not equal to $-1$.
        \item $\beta(G)=1$ and either $\Delta(G) = \{1, -1\}$ or $\Delta(G)$ is not generated by an integer.
    \end{enumerate}
\end{theorem}
\begin{proof}
    For the ``only if'' direction, this is a direct application of Lemma \ref{lem:integralModInfiniteGen} and Lemma \ref{lem:Betti2InfiniteGen}.

    For the ``if'' direction, we apply Theorem \ref{thm:phi_hyperbolic} and Proposition \ref{prop:negSgnFg} to restrict to the case $\sgn(\phi) = 1$. If $T/G$ is a tree we apply Theorem \ref{thm:contractible}. Otherwise, $T/G$ has Betti number 1 and we may apply Proposition \ref{prop:BettiNumber1fgFix}.
\end{proof}

\begin{theorem}
    \label{thm:bounded}
    Suppose $(G,T)$ is a 1-free, non-elementary GBS system. Then, for $\phi \in \aut^T(G)$:

    \begin{enumerate}
        \item If $\beta(G) = 0$, then $\rk(\fix(\phi)) \leq \max(1, 2|E(T/G)|)$.
        \item If $\beta(G) = 1$ and $\Delta(G) = \{1, -1\}$, then $\rk(\fix(\phi)) \leq 2|V(T/G)| + 1$. 
    \end{enumerate}
    
    Otherwise, there is no bound on $\rk(\fix(\phi))$.
    
\end{theorem}
\begin{proof}
    The given bounds on the rank follow from Corollary \ref{corol:contractible_combined} and Propositions \ref{prop:BettiNumber1fgFix}.

    In each other case, either there are non-finitely generated fixed subgroups by Theorem \ref{thm:characterisingFg}, or there are fixed subgroups for automorphisms $\phi \in \aut^T(G)$ of unbounded rank by Lemma \ref{lem:integralModInfiniteGen}.
\end{proof}

\begin{corollary}\label{cor:baumslagsolitar}
    For fixed $p, q \in \mathbb{Z}$ with $|q| \geq |p|$ and $|p| \neq 1$, consider the group $\bs(p,q) = \langle x, t \mid x^p = tx^qt^{-1} \rangle$.

    \begin{enumerate}
        \item If $p = -q$ then, for all $\phi \in \aut(\bs(p,q))$, $\rk(\fix(\phi)) \leq 3$.
        \item If $p \nmid q$ then, for all $\phi \in \aut(\bs(p,q))$, $\rk(\fix(\phi))$ is finite, but there is no bound on the rank.
        \item Otherwise, there exists $\phi \in \aut(\bs(p,q))$ such that $\rk(\fix(\phi))$ is infinite.
    \end{enumerate}
\end{corollary}
\begin{proof}
    In the first two cases, $\aut^T(G) = \aut(G)$ \cite[Corollary 1.2]{levitt2007gbs}. Now the result follows directly from Theorem \ref{thm:characterisingFg} and Theorem \ref{thm:bounded}.
\end{proof}

\section{Fixed subgroups of tree GBS groups}

In this section, we restrict to \emph{tree} GBS groups; that is, to those GBS groups $G$ admitting a GBS system $(G, t)$ such that $T/G$ is a tree.
For non-elementary GBS groups, this does not depend on the choice of GBS system.
This restriction allows us to use more algebraic tools, and avoid the restriction to $\aut^T(G)$.
Some of our results work in the more general setting of non-elementary GBS groups with centre, that is where $\Delta(G) = \{1\}$.
In this case the centre is infinite cyclic.

\p{BNS invariants}
We will make use of the \emph{Bieri-Neumann-Strebel (BNS) invariant} of GBS groups, which is a useful tool for determining when kernels of maps to $\mathbb{Z}$ are finitely generated. We recall the definitions and properties we will use, based on \cite{strebel2012notes}.

%\begin{definition}
    Given $G$ a finitely generated group, we write $S(G)$ for the sphere of non-trivial real-valued homomorphisms on $G$ up to positive rescaling, $$S(G) := (\hom(G, \mathbb{R}) \setminus \{0\})/\mathbb{R}_+.$$ Given $H \leq G$, we define the relative sphere as follows, $$S(G,H) := \{ [\chi] \in S(G) \mid \chi(H) = 0\}.$$
%\end{definition}

%We are now ready to define the BNS invariant.
%We use one of the equivalent definitions given by \cite{strebel2012notes}.

%\begin{definition}\label{def:BNS}
    Given a group $G$ we define the \emph{BNS invariant} $\Sigma(G)$ as a subset of $S(G)$. The definition is based on a choice of Cayley graph $\Gamma$, and we write $\Gamma_\chi$ for the induced subgraph containing only the vertices where $[\chi]$ is non-negative. Then we define, $$\Sigma(G) := \{[\chi] \in S(G) \mid \Gamma_\chi \text{ is connected}\}.$$
%\end{definition}

Note that the choice of $\Gamma$ does not change the BNS invariant, justifying the notation $\Sigma(G)$ \cite[Theorem A2.3]{strebel2012notes}. Our interest in the BNS invariant is due to the following theorem.

\begin{theorem}\cite[Theorem B1]{bieri1987geometric}\label{thm:BNSfg}
        If $G$ is finitely generated and $N \trianglelefteq G$ with $G/N$ abelian, then $N$ is finitely generated if and only if $S(G,N) \subseteq \Sigma(G)$.
\end{theorem}

We will be interested in specialising to the case where the quotient is $\mathbb{Z}$. In this case, writing $\psi$ for the quotient map, we note that $S(G, N) = \{[\psi], [-\psi]\}$.

\p{Actions on trees}
Suppose $G$ acts on a tree.
The following gives sufficient conditions for the containment $\chi: G \rightarrow \mathbb{R} \in \Sigma(G)$. 

\begin{lemma}\cite[Corollary 2.2]{cashen2016mapping}\label{lem:BNSCL}
    Let $G$ act on a tree with finite quotient such that every vertex stabiliser is finitely generated. Then let $\chi: G \rightarrow \mathbb{R}$ be a homomorphism such that, \begin{enumerate}
        \item $\chi$ is non-trivial on each edge stabiliser,
        \item For each vertex $v$, $[\chi|_{G_v}] \in \Sigma(G_v)$,
    \end{enumerate} then $[\chi] \in \Sigma(G)$.
\end{lemma}

\p{GBS groups}
We now apply the above to GBS groups.
%approach the main theorem of this section.

\begin{lemma}\label{lem:gphi}
    Let $G$ be a non-elementary GBS group with centre. Then $G_\phi := \{g \in G \mid g^{-1}\phi(g) \in Z(G) \}$ is finitely generated for all $\phi \in \aut(G)$.
\end{lemma}
\begin{proof}
    First notice, since $Z(G)$ is characteristic, $\phi$ induces an automorphism $\bar{\phi}$ on $G/Z(G)$, which is virtually free and in particular hyperbolic since it acts on a tree with finite vertex groups. It follows that $\fix(\bar{\phi})$ is finitely generated \cite{Hyperbolic}. 

    Now we notice that, by definition, $G_\phi$ is the full pre-image of $\fix(\bar{\phi})$. It follows that we have a short exact sequence, $$1 \rightarrow K \rightarrow G_\phi \rightarrow \fix(\bar{\phi}) \rightarrow 1,$$ coming from the restriction of $G \rightarrow G/Z(G)$. Finally, since $K \leq Z(G) \cong \mathbb{Z}$, $K$ is finitely generated, so $G_\phi$ is finitely generated.
\end{proof}

\begin{lemma}\label{lem:negCentreFixFg}
    Let $G$ be a non-elementary GBS group with centre. Then $\fix(\phi)$ is finitely generated for all $\phi \in \aut(G)$ acting non-trivially on $Z(G)$.
\end{lemma}
\begin{proof}
    We consider the map $\psi: G_\phi \rightarrow \mathbb{Z}$ given by $g \mapsto g^{-1}\phi(g) \in Z(G)$ (we arbitrarily choose $z$ cyclically generating $Z(G)$ and identify the centre with $\mathbb{Z}$ via $z \mapsto 1$). We note that this is a homomorphism since if $g^{-1}\phi(g) = z^i$ and $h^{-1}\phi(h) = z^j$, then $$(gh)^{-1}\phi(gh)= h^{-1}g^{-1}\phi(g)\phi(h) = h^{-1}x^i\phi(h) = x^{i+j}.$$ Clearly $\ker(\psi) = \fix(\phi)$. 

    We choose $T$ an arbitrary GBS tree for $G$, and pass to an invariant subtree $T'$ where $G_\phi$ acts minimally \cite[Corollary 7.3, Proposition 7.5]{bass1993covering}. By Proposition \ref{prop:bassFg}, $T'/G_\phi$ is finite. All edge and vertex groups are isomorphic to $\mathbb{Z}$, since $Z(G)$ is a subgroup of $G_\phi$ and is contained in every edge and vertex stabiliser in $T$.

    Since $\phi$ acts non-trivially on the centre, it follows that $\psi(z) = z^{-2}$, so $\psi$ is non-trivial on the edge groups. Furthermore, we note that all of the vertex groups are isomoprhic to $\mathbb{Z}$, and that $\Sigma(\mathbb{Z}) = S(\mathbb{Z})$  follows immediately from the definition of BNS invariants given above. By Lemma \ref{lem:BNSCL}, $[\psi] \in \Sigma(G)$. Since the choice of isomorphism $Z(G) \rightarrow \mathbb{Z}$ was arbitrary, $[-\psi] \in \Sigma(G)$ by the same argument. It follows by Theorem \ref{thm:BNSfg} that $\ker(\psi) = \fix(\phi)$ is finitely generated. 
\end{proof}

The main result of this section is as follows.
Recall that a non-elementary GBS group has Betti number $0$ if for some (and hence any) GBS system $(G, T)$, the quotient $T/G$ is a tree. It follows that $\ab(G)$ is cyclic.

\begin{theorem}\label{thm:treeFg}
    Let $G$ be a non-elementary GBS group with Betti number $0$. Then $\fix(\phi)$ is finitely generated for all $\phi \in \aut(G)$.
\end{theorem}
\begin{proof}
    If $\phi$ acts non-trivially on the centre, then we apply Lemma \ref{lem:negCentreFixFg} and we are done. 
    
    So suppose $\phi$ acts trivially on the centre. Notice that the induced automorphisms $\bar{\phi}: \ab(G) \rightarrow \ab(G)$ is trivial, since $\ab(G)$ is cyclic and the centre of course appears non-trivially. As before we consider $G_\phi$, which is finitely generated by Lemma \ref{lem:gphi}. Note that if, for arbitrary $g \in G$ and $i \in \mathbb{Z}$, we have $g^{-1}\phi(g) = z^i$, then in the abelianisation $\bar{\phi}(\bar{g}) = \bar{g}\bar{z}^i$. However, $\bar{\phi}$ was trivial, so it must be that $i = 0$. Hence, in fact $G_\phi = \fix(\phi)$, and the proof is complete.
\end{proof}

\bibliographystyle{amsalpha}
\bibliography{bibliography}

\end{document}